\documentclass{article}%
\usepackage{graphicx}
\usepackage{amsmath}
\usepackage{amsfonts}
\usepackage{amssymb}%
\providecommand{\U}[1]{\protect\rule{.1in}{.1in}}
\providecommand{\U}[1]{\protect\rule{.1in}{.1in}}
\providecommand{\U}[1]{\protect\rule{.1in}{.1in}}
\providecommand{\U}[1]{\protect\rule{.1in}{.1in}}
\providecommand{\U}[1]{\protect\rule{.1in}{.1in}}
\providecommand{\U}[1]{\protect\rule{.1in}{.1in}}
\providecommand{\U}[1]{\protect\rule{.1in}{.1in}}
\providecommand{\U}[1]{\protect\rule{.1in}{.1in}}
\providecommand{\U}[1]{\protect\rule{.1in}{.1in}}
\providecommand{\U}[1]{\protect\rule{.1in}{.1in}}
\providecommand{\U}[1]{\protect\rule{.1in}{.1in}}
\providecommand{\U}[1]{\protect\rule{.1in}{.1in}}
\providecommand{\U}[1]{\protect\rule{.1in}{.1in}}
\providecommand{\U}[1]{\protect\rule{.1in}{.1in}}
\providecommand{\U}[1]{\protect\rule{.1in}{.1in}}
\providecommand{\U}[1]{\protect\rule{.1in}{.1in}}
\providecommand{\U}[1]{\protect\rule{.1in}{.1in}}
\providecommand{\U}[1]{\protect\rule{.1in}{.1in}}
\providecommand{\U}[1]{\protect\rule{.1in}{.1in}}
\providecommand{\U}[1]{\protect\rule{.1in}{.1in}}
\providecommand{\U}[1]{\protect\rule{.1in}{.1in}}
\providecommand{\U}[1]{\protect\rule{.1in}{.1in}}
\providecommand{\U}[1]{\protect\rule{.1in}{.1in}}
\providecommand{\U}[1]{\protect\rule{.1in}{.1in}}
\providecommand{\U}[1]{\protect\rule{.1in}{.1in}}
\providecommand{\U}[1]{\protect\rule{.1in}{.1in}}
\providecommand{\U}[1]{\protect\rule{.1in}{.1in}}
\providecommand{\U}[1]{\protect\rule{.1in}{.1in}}
\providecommand{\U}[1]{\protect\rule{.1in}{.1in}}

\newtheorem{theorem}{Theorem}
{}

{}

\newenvironment{proof}[1][Proof]{\textbf{#1.} }{\ \rule{0.5em}{0.5em}}

\begin{document}

\title{On the differential operators of odd order with PT-symmetric periodic matrix coefficients}
\author{O. A. Veliev\\{\small \ }Department of Mechanical Engineering, Dogus University, Istanbul, Turkey\\\ {\small e-mail: oveliev@dogus.edu.tr}}
\date{}
\maketitle

\begin{abstract}
In this paper we investigate the spectrum of the differential operators
generated by the ordinary differential expression of odd order with
PT-symmertic periodic matrix coefficients.



\end{abstract}

In this paper, we consider the spectrum $\sigma(T)$ of the differential
operator $T$ generated in the space $L_{2}^{m}(-\infty,\infty)$ by the
differential expression
\begin{equation}
(i)^{n}y^{(n)}+(i)^{n-1}P_{1}y^{(n-1)}+(i)^{n-2}P_{2}y^{(n-2)}+...+P_{n}%
y,\tag{1}%
\end{equation}
where $P_{k}$ for $k=1,2,...,n$ are the $m\times m$ matrices with the
complex-valued PT-symmetric periodic entries
\begin{equation}
p_{k,i,j}\left(  x+1\right)  =p_{k,i,j}\left(  x\right)  ,\text{ }%
p_{k,i,j}\left(  -x\right)  =\overline{p_{k,i,j}\left(  x\right)  }\tag{2}%
\end{equation}
and $y=(y_{1},y_{2},...,y_{m})^{T}$ is a vector-valued function. Here
$L_{2}^{m}(a,b)$ for $-\infty\leq a<b\leq\infty$ is the space of the
vector-valued functions $f=\left(  f_{1},f_{2},...,f_{m}\right)  ^{T}$ with
the norm $\left\Vert \cdot\right\Vert _{(a,b)}$ and inner product
$(\cdot,\cdot)_{(a,b)}$ defined by%
\[
\left\Vert f\right\Vert _{(a,b)}^{2}=\int_{a}^{b}\left\vert f\left(  x\right)
\right\vert ^{2}dx,\text{ }(f,g)_{(a,b)}=\int_{a}^{b}\left\langle f\left(
x\right)  ,g\left(  x\right)  \right\rangle dx,
\]
where $\left\vert \cdot\right\vert $ and $\left\langle \cdot,\cdot
\right\rangle $ are the norm and inner product in $\mathbb{C}^{m}.$ We prove
that if $m$ and $n$ are the odd numbers, then $\mathbb{R}\subset\sigma(T).$

Note that there are a large number of papers for the scalar case $m=1$ and
$n=2$, namely for the Schr\"{o}dinger operator (see the monographs [1,
Chapters 4 and 6] and [5, Chapters 3 and 5] and the papers they refer to).
However, as far as I know, only paper [6] is devoted to the vector
Schr\"{o}dinger operator with a PT-symmetric periodic matrix potential, where
$n=2$. The main result of this paper concerns the spectrum for the case when
$n$ and $m$ are odd numbers. Moreover, this result and the method used in this
paper are completely different from the results and methods of those papers.
That is why, and in order not to deviate from the purpose of this very short
paper, we do not discuss them here.

First, using (1) and (2) we prove the following theorem about the solutions of
the equation
\begin{equation}
(i)^{n}y^{(n)}+(i)^{n-1}P_{1}y^{(n-1)}+(i)^{n-2}P_{2}y^{(n-2)}+...+P_{n}%
y=\lambda y, \tag{3}%
\end{equation}
where $n$ and $m$ are arbitrary positive integers.

\begin{theorem}
$(a)$ If $\Psi$ is a solution of (3), then the function $\Phi$ defined by
$\Phi(x,\lambda)=\overline{\Psi(-x,\lambda)}$ is a solution of
\begin{equation}
(i)^{n}y^{(n)}+(i)^{n-1}P_{1}y^{(n-1)}+(i)^{n-2}P_{2}y^{(n-2)}+...+P_{n}%
y=\overline{\lambda}y.\tag{4}%
\end{equation}

$(b)$ Suppose that $\lambda$ is a real number. Let $S_{+}(\lambda)$ and
$S_{-}(\lambda)$ be the spaces of the solutions of (3) belonging to $L_{2}%
^{m}(0,\infty)$ and $L_{2}^{m}(-\infty,0)$ respectively. Then the mapping $A$
defined by $A:f(x)\rightarrow\overline{f(-x)}$ is an antilinear isometric
bijection from $S_{+}(\lambda)$ onto $S_{-}(\lambda).$
\end{theorem}

\begin{proof}
$(a)$ If $\Psi(x,\lambda)$ is a solution of (3) then the equality
\[
(i)^{n}\Psi^{(n)}(x)+(i)^{n-1}P_{1}(x)\Psi^{(n-1)}(x)+...+P_{n}(x)\Psi
(x)=\lambda\Psi(x)
\]
holds for all $x\in(-\infty,\infty).$ In this equality replacing $x$ by $-x$,
taking the conjugate and using (2) we obtain
\[
(-i)^{n}\overline{\Psi^{(n)}(-x)}+(-i)^{n-1}P_{1}(x)\overline{\Psi
^{(n-1)}(-x)}+...+P_{n}(x)\overline{\Psi(-x)}=\overline{\lambda}\overline
{\Psi(-x)}.
\]
On the other hand, it follows from the definition of $\Phi$ that $\Phi
^{(k)}(x)=(-1)^{k}\overline{\Psi^{(k)}(-x)}.$ Therefore, we have
\[
(i)^{n}\Phi^{(n)}(x)+(i)^{n-1}P_{1}(x)\Phi^{(n-1)}(x)+...+P_{n}(x)\Phi
(x)=\overline{\lambda}\Phi(x).
\]
It means that $\Phi$ is the solution of (4).

$(b)$ If $\lambda$ is a real number, then it follows from $(a)$ that if
$\Psi(x,\lambda)$ is a solution of (3), then $\Phi(x,\lambda)$ is also a
solution of (3). Moreover, using the equality $\Phi(x)=\overline{\Psi(-x)}$
and the substituting $u=-x$ we get
\begin{equation}%
{\displaystyle\int\limits_{-a}^{0}}
\left\vert \Phi(x)\right\vert ^{2}dx=%
{\displaystyle\int\limits_{-a}^{0}}
\left\vert \Psi(-x)\right\vert ^{2}dx=%
{\displaystyle\int\limits_{0}^{a}}
\left\vert \Psi(u)\right\vert ^{2}du\tag{5}%
\end{equation}
for any $0<a<\infty.$ If $\Psi\in S_{+}(\lambda),$ then using (5) and letting
$a$ tend to infinity we obtain that $\Phi\in S_{-}(\lambda).$ Thus for any
solution of (3) belonging to $L_{2}^{m}(0,\infty)$ there corresponds a
solution $\Phi$ of (3) belonging to $L_{2}^{m}(-\infty,0).$ In the same way we
prove that, for any solution $\Psi$ of (3) belonging to $L_{2}^{m}(-\infty,0)$
there corresponds a solution $\Phi$ of (3) belonging to $L_{2}^{m}(0,\infty).$
Moreover, in the both cases these correspondences are made by the mapping $A.$
Thus the mapping $A$ is an antilinear isometric bijection from $S_{+}%
(\lambda)$ onto $S_{-}(\lambda).$
\end{proof}

Now using Theorem 1, the following well-known connection of $T$ with the
operators defined in $L_{2}^{m}[0,1]$ by the quasiperiodic boundary conditions
and the well-known Floquet theory we consider the spectrum of $T.$ It is well
known that (see for example [2, 4]) the spectrum $\sigma(T)$ of $T$ is the
union of the spectra $\sigma(T_{t})$ of the operators $T_{t}$ for $t\in
\lbrack0,2\pi)$ generated in $L_{2}^{m}[0,1]$ by (1) and the boundary
conditions
\begin{equation}
y^{(\mathbb{\nu})}\left(  1\right)  =e^{it}y^{(\mathbb{\nu})}\left(  0\right)
,\text{ }\mathbb{\nu}=0,1,...,(n-1).\tag{6}%
\end{equation}
Thus
\begin{equation}
\sigma(T)=%
{\textstyle\bigcup\limits_{t\in\lbrack0,2\pi)}}
\sigma(T_{t}).\tag{7}%
\end{equation}
The spectrum of $T_{t}$ consists of the eigenvalues that are the roots of the
characteristic equation
\begin{equation}
\det(Y_{j}^{(\nu-1)}(1,\lambda)-e^{it}Y_{j}^{(\nu-1)}(0,\lambda))_{j,\nu
=1}^{n}=0\tag{8}%
\end{equation}
of (3), where $Y_{1}(x,\lambda),Y_{2}(x,\lambda),\ldots,Y_{n}(x,\lambda)$ are
the solutions of the matrix equation
\[
(i)^{n}Y^{(n)}(x)+(i)^{n-1}P_{1}(x)Y^{(n-1)}(x)+...+P_{n}(x)Y=\lambda Y(x)
\]
satisfying $Y_{k}^{(j)}(0,\lambda)=0_{m}$ for $j\neq k-1$ and $Y_{k}%
^{(k-1)}(0,\lambda)=I_{m}$, where $0_{m}$ and $I_{m}$ are $m\times m$ zero and
identity matrices respectively (see [3, Chapter 3], and [7]).

Now, let us recall the well-known Floquet theory for the differential
equations with periodic coefficients. The $n$-th order equation (3) can be
reduced to the first-order equation
\begin{equation}
\mathbf{x}^{^{\prime}}=A(x,\lambda)\mathbf{x}\tag{9}%
\end{equation}
by setting $\mathbf{x=(x}_{1}\mathbf{,x}_{2}\mathbf{,...x}_{n}\mathbf{)}^{T},$
where $\mathbf{x}_{1}=y\mathbf{,x}_{2}=y^{^{\prime}}\mathbf{,...x}%
_{n}=y^{(n-1)},$ $A(x,\lambda)$\ is $mn\times mn$ matrix and $A(x+1,\lambda
)=A(x,\lambda).$ According to the Floquet theory, any solution (3) is a linear
combination of solutions (3) having the form
\begin{equation}
e^{it_{k}x}(p_{0,k}(x)+xp_{1,k}(x)+...+x^{s_{k}}p_{s_{k},k}(x)),\tag{10}%
\end{equation}
where $p_{0,k}(x),p_{1,k}(x),...$ are the periodic functions, $e^{it_{1}%
},e^{it_{2}},...e^{it_{nm}}$ are the multipliers of (3) and (9). These
multipliers are the roots of the characteristic equation
\begin{equation}
\det(X(1,\lambda)-e^{it}I)=0\tag{11}%
\end{equation}
of (9), where $X(x,\lambda)$ is a solution of the matrix equation
$X^{^{\prime}}(x,\lambda)=A(x,\lambda)X(x,\lambda\mathbf{)}$ satisfying the
initial condition $X(0,\lambda)=I_{mn}$ (see \ for example [8, Chapter 2]).
Using (6)-(9) one can easily verify that the characteristic equations (8) and
(11) of (3) and (9) are the same and
\begin{equation}
\lambda\in\sigma(T)\Longleftrightarrow\exists t\in\mathbb{R}:e^{it}\in
M(\lambda)\Longleftrightarrow\exists\Psi\in S(\lambda):\Psi(x,\lambda
)=e^{itx}p(x),\tag{12}%
\end{equation}
where $M(\lambda)$ and $S(\lambda)$ are respectively the sets of the
multipliers and solutions of (3) and $p(x+1)=p(x).$ Now, using this and
Theorem 1 we prove the following.

\begin{theorem}
$(a)$ The nonreal part $\sigma(T)\backslash\mathbb{R}$ of $\sigma
(T)$\ consists of the pairs of curves symmetric with respect to the real line.

$(b)$ If $n$ and $m$ are the odd numbers, then $\mathbb{R}\subset\sigma(T)$.
\end{theorem}

\begin{proof}
$(a)$ Since the spectrum of $T$ consists of the union of the curves it is
enough to prove that if $\lambda\in\sigma(T),$ then $\overline{\lambda}%
\in\sigma(T).$ If $\lambda\in\sigma(T),$ then by (12) equation (3) has a
solution $\Psi(x,\lambda)$ of the form $e^{itx}p(x)$, where $t\in\mathbb{R}$
and $p(x+1)=p(x).$ Then, by Theorem 1$(a),$ $\Phi(x,\lambda)=\overline
{\Psi(-x,\lambda)}$ is a solution of (4) and $\Phi(x,\lambda)=e^{itx}%
\overline{p(-x)}$ . Thus, by (12), $\overline{\lambda}\in\sigma(T).$

$(b)$ Suppose to the contrary that there exists a real number $\lambda$ such
that $\lambda\notin\sigma(T).$ Then, it follows from (12) that the absolute
values of all multipliers $e^{it_{1}},$ $e^{it_{2}},...,e^{it_{nm}}$ of (3)
differ from $1,$ that is, $t_{k}\neq\overline{t_{k}}$ for all $k=1,2,...,mn.$
It is clear that solution (10) belong to $L_{2}^{m}(0,\infty)$ and $L_{2}%
^{m}(-\infty,0)$ respectively, if $\operatorname{Im}t_{k}>0$ \ and
$\operatorname{Im}t_{k}<0.$ Therefore, the set $S(\lambda)$ of all solution of
(3) is
\[
S(\lambda)=\left\{  y+z:y\in S_{+}(\lambda),\text{ }z\in S_{-}(\lambda
)\right\}  ,
\]
where $S_{+}(\lambda)$ and $S_{-}(\lambda)$ are defined in Theorem 1$(b)$ and
$S_{+}(\lambda)\cap S_{-}(\lambda)=\left\{  0\right\}  .$ This implies that
$\dim S(\lambda)=\dim S_{+}(\lambda)+\dim S_{-}(\lambda).$ On the other hand,
by Theorem 1$(b),$ $\dim S_{+}(\lambda)=\dim S_{-}(\lambda).$ Therefore, the
dimension $nm$ of the space of all solutions of (3) is $2\dim S_{+}(\lambda)$
which contradicts the assumption that $n$ and $m$ are odd numbers.
\end{proof}

\end{document}